\documentclass[11pt,a4paper]{article}
\usepackage[hmargin=20mm,vmargin=20mm]{geometry}
\usepackage{amsmath,amssymb,amsthm}
\usepackage{microtype}
\usepackage{xcolor}
\usepackage[colorlinks=true,linkcolor=black,citecolor=black,urlcolor=blue!55!black]{hyperref}
\hypersetup{pdftitle={Contractibility of the Gromov--Hausdorff cloud of a normed space},
  pdfauthor={I. N. Mikhailov},pdfsubject={Metric geometry}}
\newcommand{\R}{\mathbb R}
\newcommand{\Z}{\mathbb Z}
\newcommand{\gh}{\mathrm{GH}}
\newcommand{\hd}{\mathrm H}
\DeclareMathOperator{\dis}{dis}
\DeclareMathOperator{\dist}{dist}

\newtheorem{theorem}{Theorem}
\newtheorem{lemma}[theorem]{Lemma}
\newcommand{\arxiv}[1]{\href{https://arxiv.org/abs/#1}{arXiv:#1}}
\title{Contractibility of the Gromov--Hausdorff cloud\\ of a normed space}
\author{I. N. Mikhailov\\[3pt]
  \small Faculty of Mechanics and Mathematics, Lomonosov Moscow State University, Moscow, Russia}
\date{}
\begin{document}
\maketitle
\begin{abstract}
We prove that the Gromov--Hausdorff cloud of every real or complex
normed space $X$ admits a strong deformation retraction onto the point
represented by $X$. The result includes incomplete normed spaces.
\end{abstract}

\section{Introduction}

M.~Gromov considered the moduli space of noncompact metric spaces at
finite Gromov--Hausdorff distance from a given space $X$. He stated that
these moduli spaces are complete and contractible, giving $X=\R^n$ as
an example \cite[p.~78]{Gromov}. S.~A.~Bogatyy and A.~A.~Tuzhilin
called the classes of spaces at finite distance from one another
\emph{clouds} and proved their completeness \cite{BT}. S.~A.~Bogatyy
later proved that sparse clouds are contractible \cite{sparse}.
We write $[X]$ for the cloud of $X$, identify spaces at zero
Gromov--Hausdorff distance, and denote the point represented by $X$
again by $X$.

For a normed space $X$, the spaces $\lambda X$ and $X$ are isometric
for all $\lambda>0$, where $\lambda X$ means that all distances are
multiplied by $\lambda$. Hence $d_{\gh}(\lambda Y,X)=\lambda d_{\gh}(Y,X)$
for $Y\in[X]$, and $\lambda Y\to X$ as $\lambda\downarrow0$
\cite[Definition~3.7]{BT}. A contraction also requires continuity
at positive scales \cite{centers}. This can fail even in $[\R^n]$:
$d_{\gh}(\Z^n,\lambda\Z^n)\ge 1/2$ for every $\lambda>1$ \cite{lattice}.

In the first version of this note \cite{nets}, we proved that the space
of $\varepsilon$-nets in $\R$ is contractible for both the Hausdorff
and Gromov--Hausdorff distances. The second version \cite{hilbert-v2}
proved contractibility of the cloud of a real Hilbert space.
In this third version we extend the result to arbitrary normed spaces.

\begin{theorem}\label{thm:normed}
For every real or complex normed space $X$, the cloud $[X]$ admits a
strong deformation retraction onto the point $X$.
\end{theorem}

The contraction $\Phi:[X]\times[0,1]\to[X]$ is jointly continuous for
the Gromov--Hausdorff metric on the cloud and the usual metric on the
interval. It satisfies $\Phi(Y,0)=Y$, $\Phi(Y,1)=X$ and $\Phi(X,t)=X$.

The main change in the proof is the choice of distance functions at
infinity. For example, in $\ell_1^2$ the limit of the normalized distance
from $(a,b)$ to $(t,0)$ as $t\to+\infty$ is $-a+|b|$. To recover
linearity for an arbitrary norm, we use limits that are insensitive
to bounded shifts of the observation point. These limits give a
linear isometric copy of the Banach completion. For each $Y\in[X]$,
we first deform $Y$ to its image under the distance profiles, then
fill in the closed convex hull of that image. Finally we contract
this hull to the copy of the Banach completion.

\section{Preliminaries}

\subsection*{Metric conventions}
All metric spaces are nonempty. A correspondence between $Y$ and $Z$
is a set $Q\subset Y\times Z$ whose two coordinate projections are onto.
Its distortion and the Gromov--Hausdorff distance are given by
\begin{equation}\label{eq:correspondence}
 \dis Q=\sup_{(y,z),(y',z')\in Q}
    |d_Y(y,y')-d_Z(z,z')|,
 \qquad
 d_{\gh}(Y,Z)=\tfrac12\inf_Q\dis Q .
\end{equation}
Here $d_{\gh}$ is the unpointed Gromov--Hausdorff distance \cite{BT}.
For subsets $A,B$ of a metric space, write $\dist(x,B)=\inf_{b\in B}d(x,b)$
and
\[
 d_{\hd}(A,B)=\max\big\{\sup_{a\in A}\dist(a,B),\ \sup_{b\in B}\dist(b,A)\big\}.
\]
In a common metric space $d_{\gh}(A,B)\le d_{\hd}(A,B)$. Thus a dense
subspace and its closure represent the same point of a cloud.
For a pseudometric $p$ on $Y$, its \emph{metric quotient} identifies
$x$ and $y$ whenever $p(x,y)=0$. A correspondence between two sets with
pseudometrics induces a correspondence of the same distortion between
their metric quotients, since the quotient maps preserve every distance.
One may take all metric spaces in a fixed universe and form their clouds
in a larger universe.

\subsection*{Ultrafilters}

A \emph{proper filter} on a set $S$ is a family of subsets of $S$
that contains $S$, excludes the empty set, and is closed under
finite intersections and passage to supersets. An \emph{ultrafilter}
is a maximal proper filter. Equivalently, for each $A\subset S$ it
contains exactly one of $A$ and $S\setminus A$. We use the fact that
any family of subsets with the \emph{finite intersection property}
(every finite subfamily has nonempty intersection) is contained
in an ultrafilter. Indeed, the family generates a proper filter,
which Zorn's lemma extends to a maximal one.

For a real-valued function $h$ on $S$, the notation
$\lim_\omega h(s)=a$ means that $\{s:|h(s)-a|<\varepsilon\}\in\omega$
for every $\varepsilon>0$. If $h$ is bounded, the closed sets
$\overline{h(A)}$ for $A\in\omega$ have the finite intersection property
in a compact interval. A point of their intersection is an
ultrafilter limit, by the either-or property. Disjoint
neighborhoods of distinct points give uniqueness. These limits
preserve sums and inequalities. We write $h(s)\to\infty$
along $\omega$ if $\{s:h(s)>R\}\in\omega$ for every $R>0$.
For instance, an ultrafilter on $\mathbb N$ containing all tails
gives the usual limit of every convergent sequence.

For a map $f:S\to T$, the \emph{pushforward} $f_*\omega$ is the
ultrafilter on $T$ defined by $A\in f_*\omega$ if and only if
$f^{-1}(A)\in\omega$. In particular,
$\lim_{f_*\omega}h=\lim_{\omega}(h\circ f)$ for bounded $h$ on $T$.
If $f$ is onto, any ultrafilter on $T$ lifts to one on $S$ by
pushing it forward along a section of $f$.

\subsection*{Reduction to Banach spaces}
Let $B$ be the Banach completion of $X$. Since $X$ is dense in $B$,
$d_{\hd}(X,B)=0$ in $B$, and hence $d_{\gh}(X,B)=0$. Thus $[X]=[B]$
and $X$ and $B$ are the same point of this cloud. It suffices to
construct a strong deformation retraction onto $B$. For a complex
space, we use the underlying real Banach space with the same norm.

If $B=\{0\}$, the cloud consists of bounded metric spaces. Set
$\Phi(Y,t)=(1-t)Y$, with the one-point space at $t=1$. A correspondence
of distortion $\delta$ between $Y$ and $Z$ induces a correspondence
between $\Phi(Y,s)$ and $\Phi(Z,t)$ of distortion at most
$|t-s|\operatorname{diam}Y+(1-t)\delta$ by comparing rescaled distances.
Taking the infimum over correspondences gives
\[
 d_{\gh}(\Phi(Z,t),\Phi(Y,s))
 \le (1-t)d_{\gh}(Z,Y)+\tfrac12|t-s|\operatorname{diam}Y.
\]
For fixed $Y,s$, the right-hand side tends to zero as $(Z,t)\to(Y,s)$.
The one-point space is fixed and every $Y$ reaches it at $t=1$.

\section{Distance profiles on a Banach space}

From now on $B$ is a nonzero real Banach space. We want to choose
observation points $z$ far from the origin so that translating them
by a bounded amount hardly changes their distance increments.
For $v,z\in B$ and $M>0$, put
\[
 \Delta_v(z)=\|z+v\|-\|z\|,\qquad
 F_{v,M}(z)=\sup_{\|b\|\le M}|\Delta_v(z+b)-\Delta_v(z)|.
\]
We call an ultrafilter $\omega$ on $B$ \emph{admissible} if along it
$\|z\|\to\infty$ and $F_{v,M}(z)\to0$ for every $v\in B$ and $M>0$.
Write $\Omega_B$ for the set of these ultrafilters.

\begin{lemma}\label{lem:linear-profiles}
For $\omega\in\Omega_B$, the map $L_\omega(v)=\lim_\omega\Delta_v(z)$
is a real linear functional of norm at most $1$. Moreover, $\Omega_B$
is nonempty and
\[
 \sup_{\omega\in\Omega_B}|L_\omega(v)|=\|v\|\qquad(v\in B).
\]
\end{lemma}
\begin{proof}
The identity $\Delta_{u+v}(z)=\Delta_u(z+v)+\Delta_v(z)$ gives
$L_\omega(u+v)=L_\omega(u)+L_\omega(v)$, since
$|\Delta_u(z+v)-\Delta_u(z)|\le F_{u,\|v\|+1}(z)$.
The bound $|\Delta_v(z)|\le\|v\|$ makes the additive map continuous,
hence real linear, with norm at most $1$.

Fix $x\ne0$ and $\eta>0$. We show that the requirements for
admissibility and the additional condition
$\Delta_x(z)>\|x\|-\eta$ have the finite intersection property.
A finite list of requirements involves only finitely many vectors $v$
and positive thresholds for the $F_{v,M}$. Let $E$ be the span of those
vectors and $x$. Take an open ball in $E$ about $x$ of radius less
than $\min\{\|x\|,\eta/2\}$. It lies away from $0$. The norm on $E$
is Lipschitz in Euclidean coordinates, so Rademacher's theorem gives
a point $a$ in the ball at which the restricted norm is differentiable.
Its derivative $\ell$ satisfies $\|\ell\|\le1$ and $\ell(a)=\|a\|$
by homogeneity.
It follows that $\ell(x)\ge\|x\|-2\|a-x\|>\|x\|-\eta$.

For $v\in E$ and $M>0$ we claim that $\Delta_v(ta+b)\to\ell(v)$
as $t\to\infty$, uniformly for $b\in B$ with $\|b\|\le M$.
Set $h=\sqrt{t}$ and $g_b(s)=\|ta+b+sv\|$. For $t\ge1$, convexity gives
\[
 \frac{g_b(0)-g_b(-h)}h
 \le g_b(1)-g_b(0)
 \le \frac{g_b(h)-g_b(0)}h .
\]
Replacing $b$ by $0$ changes either outer quotient by at most $2M/h$.
For $b=0$, homogeneity writes the right quotient as
$(\|a+(h/t)v\|-\|a\|)/(h/t)$, and the left quotient is its
backward counterpart. Both tend to $\ell(v)$. This proves the
uniform convergence for shifts in the whole space $B$.

For large $t$, $z=ta$ is as far from the origin as required.
The uniform convergence gives $F_{v,M}(ta)\to0$ for each condition
in the finite list. It also gives $\Delta_x(ta)\to\ell(x)$, which
exceeds $\|x\|-\eta$. The finite intersection property gives an ultrafilter
$\omega\in\Omega_B$ with $L_\omega(x)\ge\|x\|-\eta$. Since $\eta>0$
was arbitrary, and each $L_\omega$ has norm at most $1$, the claimed
norm formula follows.
\end{proof}

\paragraph{Changes of observation points.}
If $\|z'-z\|\le R$, then $F_{v,M}(z')\le 2F_{v,M+R}(z)$ by the
triangle inequality. Also $|\|z'\|-\|z\||\le R$. Consequently a
uniformly bounded change of the observation points preserves
admissibility in either direction. If $U$ is a surjective linear
isometry, then $F_{v,M}(Uz)=F_{U^{-1}v,M}(z)$, so surjective
affine isometries preserve admissibility as well.

\section{Distance profiles on the cloud}

Call $f:Y\to B$ a \emph{comparison map} if it has finite distortion
and its image is at finite Hausdorff distance from $B$. We compute
the distortion of a map by applying the supremum in
\eqref{eq:correspondence} to its graph. If $Y\in[B]$,
selecting partners in a correspondence of distortion $\delta<\infty$
gives a comparison map $f$. Its image is $\delta$-dense in $B$
because two partners of the same point are at distance at most $\delta$.

Given a comparison map $f:Y\to B$, let $\Omega_Y$ consist of the
ultrafilters $\omega$ on $Y$ for which $f_*\omega\in\Omega_B$.
This family is nonempty. Choose $p:B\to Y$ with $\|fp(b)-b\|$
uniformly bounded. For any $\theta\in\Omega_B$, we have
$(fp)_*\theta\in\Omega_B$ by invariance under bounded changes, so
$p_*\theta\in\Omega_Y$.

\begin{lemma}\label{lem:matching}
The family $\Omega_Y$ is independent of the comparison map.
If $Q$ is a finite-distortion correspondence between $Y,Z\in[B]$,
then the two projections of an ultrafilter on $Q$ belong to
$\Omega_Y$ and $\Omega_Z$ simultaneously.
\end{lemma}

\begin{proof}
Let $f,g:Y\to B$ be comparison maps. Choose $p:B\to Y$ with
$\|fp(b)-b\|\le R$ for some $R<\infty$. For $b,b'\in B$,
\[
\begin{aligned}
 \bigl|\|gp(b)-gp(b')\|-\|b-b'\|\bigr|
 &\le \bigl|\|gp(b)-gp(b')\|-\|fp(b)-fp(b')\|\bigr|+2R\\
 &\le \dis g+\dis f+2R .
\end{aligned}
\]
Also $d_Y(pf(y),y)\le R+\dis f$, so
$\|gpf(y)-g(y)\|\le R+\dis f+\dis g$. Every point of $B$ is
within a fixed distance of $g(Y)$, and this last bound puts
every point of $g(Y)$ within a fixed distance of $gp(B)$.
Thus $gp(B)$ is coarsely dense in $B$.

The \v{S}emrl--V\"ais\"al\"a theorem \cite[Theorem~3.5]{SV} states that
a map of finite distortion between real Banach spaces with coarsely
dense image is uniformly close to a surjective affine isometry.
The affine form follows by first subtracting the value at $0$.
Applying the theorem to $gp$ gives such an isometry $T:B\to B$.
Since $gpf(y)$ is uniformly close to both $g(y)$ and $Tf(y)$,
the invariance under affine isometries and bounded changes gives
the same $\Omega_Y$ for $f$ and $g$.

Now select $q:Y\to Z$ through $Q$, and write $\delta=\dis Q$.
The map $q$ has distortion at most $\delta$ and $\delta$-dense
image, so $gq$ is a comparison map on $Y$. On $Q$ the observations
$g(q(y))$ and $g(z)$ differ by at most $\delta+\dis g$, because
$(y,q(y))$ and $(y,z)$ lie in $Q$. Thus the pushforwards of any
ultrafilter on $Q$ by $(y,z)\mapsto g(q(y))$ and $(y,z)\mapsto g(z)$
are admissible simultaneously. Using $gq$ to compute $\Omega_Y$
gives the projection assertion.
\end{proof}

Fix a basepoint $o\in Y$. The Kuratowski embedding based at $o$ sends
$y$ to the bounded function $d_Y(y,\cdot)-d_Y(o,\cdot)$ on $Y$.
Apply to this function the map that takes its ultrafilter limits
over $\Omega_Y$. Thus, in the real Banach space
$E_Y=\ell^\infty(\Omega_Y)$, put
\[
 k_y^o(\omega)=\lim_{z\to\omega}
                 \bigl(d_Y(y,z)-d_Y(o,z)\bigr),\qquad
 A_Y=\overline{\{k_y^o:y\in Y\}},\qquad
 C_Y=\overline{\operatorname{conv}}(A_Y).
\]
Here the bars denote norm closures. Replacing $o$ by $o'$ replaces
$k_y^o$ by $k_y^o-k_{o'}^o$, so it translates $A_Y$ and $C_Y$ by
the same vector. The pseudometric
\begin{equation}\label{eq:rho}
 \rho_Y(x,y)=\|k_x^o-k_y^o\|_\infty\le d_Y(x,y)
\end{equation}
is independent of $o$. The bound follows from the triangle inequality
in $Y$ before taking the limits. For $Y=B$ with $o=0$, we have
$k_y^0(\omega)=L_\omega(-y)$. Lemma~\ref{lem:linear-profiles}
therefore makes $y\mapsto k_y^0$ a linear isometry. Its image is
closed by completeness, so $A_B=C_B$ is a linear isometric copy of
$B$ and $\rho_B=d_B$.

\begin{lemma}\label{lem:ambient}
Let $Y,Z\in[B]$, and let $Q$ be a correspondence between $Y$ and $Z$
of finite distortion $\delta$. Take basepoints $(o_Y,o_Z)\in Q$.
There are isometric embeddings $j_Y:E_Y\to W$ and $j_Z:E_Z\to W$
into a real Banach space $W$, with $j_Y$ affine and $j_Z$ linear, such that
\begin{align}
 \|j_Y(k_y^{o_Y})-j_Z(k_z^{o_Z})\|
 &\le\delta \qquad ((y,z)\in Q),\label{eq:profile}\\
 d_{\hd}(j_Y(A_Y),j_Z(A_Z)),\quad d_{\hd}(j_Y(C_Y),j_Z(C_Z))
 &\le\delta.\label{eq:sets}
\end{align}
\end{lemma}
\begin{proof}
Let $\Omega_Q$ be the ultrafilters on $Q$ whose projections belong
to $\Omega_Y$ and $\Omega_Z$. Given $\omega\in\Omega_Y$, lift it
along a section of $Q\to Y$. The resulting ultrafilter projects to
$\omega$ and its other projection belongs to $\Omega_Z$ by
Lemma~\ref{lem:matching}. This shows that $\Omega_Q$ is nonempty and
projects onto $\Omega_Y$. The same construction works for $\Omega_Z$.
Pullback along these surjections gives linear isometries
$i_Y:E_Y\to\ell^\infty(\Omega_Q)$ and $i_Z:E_Z\to\ell^\infty(\Omega_Q)$.
Set $j_Y=i_Y+h$ and $j_Z=i_Z$, where $h$ is the function
\[
 h(\theta)=\lim_{(u,v)\to\theta}
   \bigl(d_Y(o_Y,u)-d_Z(o_Z,v)\bigr).
\]
Since $(o_Y,o_Z)\in Q$, we have $\|h\|_\infty\le\delta$.
For $(y,z)\in Q$, the terms involving the basepoints cancel in
\[
 \bigl[j_Y(k_y^{o_Y})-j_Z(k_z^{o_Z})\bigr](\theta)
 =\lim_{(u,v)\to\theta}\bigl(d_Y(y,u)-d_Z(z,v)\bigr).
\]
Distortion bounds the right-hand side by $\delta$ in absolute value,
giving \eqref{eq:profile}. The projections of $Q$ are onto, so
taking closures gives the bound for the two profile sets.
For the hulls, match the terms of a finite convex combination on
one side with their partners on the other. Use the same coefficients
$\alpha_i$ on both sides. By \eqref{eq:profile} the distance of the
two combinations is at most $\sum_i\alpha_i\delta=\delta$.
Affine maps preserve convex combinations, so taking closures gives
the bound for the two realized convex hulls.
\end{proof}

\newpage
\section{The three steps of the contraction}

Let $B$ be a nonzero real Banach space and $Y\in[B]$. Choose a
correspondence between $Y$ and $B$ of finite distortion $D$, and a
basepoint $o_Y$ paired with $0\in B$. By Lemma~\ref{lem:ambient},
we can realize $A_Y,C_Y$ together with a linear isometric copy $V$
of $B$ in a common Banach space. We use the same symbols for the sets
and their realized images. In this realization,
\begin{equation}\label{eq:width}
 d_{\hd}(A_Y,V)\le D,\qquad d_{\hd}(C_Y,V)\le D .
\end{equation}
Changing the basepoint translates $A_Y$ and $C_Y$ together. It therefore
preserves distances to $A_Y$ inside $C_Y$ and translates each rescaled
copy of $C_Y$. The three constructions below are thus independent of the
basepoint up to isometry.

\subsection*{Step 1. Deform the metric on \(Y\) to the metric on \(A_Y\)}

For $0\le t\le1$, let $F_t(Y)$ be the metric quotient of $Y$ for the
pseudometric
\begin{equation}\label{eq:interpolate}
 d_{Y,t}(x,y)=(1-t)d_Y(x,y)+t\rho_Y(x,y).
\end{equation}

\begin{proof}
At $t=0$ we have $F_0(Y)=Y$. At $t=1$, the map $y\mapsto k_y^o$
identifies the metric quotient with a dense subset of $A_Y$,
so $F_1(Y)=A_Y$ in the cloud.

Let $Q$ be a correspondence between $Y$ and $Z$ of finite
distortion $\delta$. Use the common realization of
Lemma~\ref{lem:ambient} and abbreviate the realized profiles
by $k_x,k_y,k_u,k_v$. For $(x,u),(y,v)\in Q$, the reverse
triangle inequality gives
\[
 \bigl|\rho_Y(x,y)-\rho_Z(u,v)\bigr|
 \le\|(k_x-k_y)-(k_u-k_v)\|
 \le\|k_x-k_u\|+\|k_y-k_v\|\le 2\delta.
\]
Hence $Q$ induces a correspondence between the metric quotients at
time $t$ with distortion at most $(1-t)\delta+2t\delta=(1+t)\delta$.
In the correspondence between $Y$ and $B$, the two metrics $d_Y$
and $\rho_Y$ differ from $d_B=\rho_B$ on matched pairs by at most
$D$ and $2D$, respectively. Thus $0\le d_Y(x,y)-\rho_Y(x,y)\le3D$.
The identity correspondence on $Y$ then has distortion at most
$3D|t-s|$ between the two interpolated metrics. Taking the infimum
over $Q$ for the space estimate gives
\begin{equation}\label{eq:F-bounds}
 d_{\gh}(F_t(Y),F_t(Z))\le(1+t)d_{\gh}(Y,Z),\qquad
 d_{\gh}(F_t(Y),F_s(Y))\le\tfrac32D|t-s|.
\end{equation}
\end{proof}

\subsection*{Step 2. Enlarge \(A_Y\) to its closed convex hull \(C_Y\)}

For $r\ge0$, take the closed $r$-neighborhood of $A_Y$ inside $C_Y$.
Write
\[
 B_r(Y)=\{c\in C_Y:\dist(c,A_Y)\le r\},\qquad
 w_Y=d_{\hd}(A_Y,C_Y)=\sup_{c\in C_Y}\dist(c,A_Y).
\]
By the triangle inequality through $V$ in \eqref{eq:width},
$w_Y\le2D$. Put $G_t(Y)=B_{t w_Y}(Y)$ for
$0\le t\le1$. Since $A_Y$ is closed and every point of $C_Y$ is at
distance at most $w_Y$ from $A_Y$, we have $G_0(Y)=A_Y$ and $G_1(Y)=C_Y$.
We prove the estimates
\begin{equation}\label{eq:G-bounds}
 d_{\gh}(G_t(Y),G_t(Z))\le10d_{\gh}(Y,Z),\qquad
 d_{\gh}(G_t(Y),G_s(Y))\le w_Y|t-s|.
\end{equation}
\begin{proof}
For $c\in C_Y$ and $r\ge0$ we have
\begin{equation}\label{eq:distance-neighborhood}
 \dist(c,B_r(Y))=\max\{0,\dist(c,A_Y)-r\}.
\end{equation}
Indeed, for $b\in B_r(Y)$ the triangle inequality gives
$\dist(c,A_Y)\le\|c-b\|+r$. For the reverse bound, take any $a\in A_Y$.
On $[a,c]$ choose the point $b$ at distance $\min\{r,\|c-a\|\}$ from
$a$. Convexity gives $b\in C_Y$, so $b\in B_r(Y)$, and
$\|c-b\|=\max\{0,\|c-a\|-r\}$. Take the infimum over $a$.

Use Lemma~\ref{lem:ambient} to put both pairs of sets in a common
Banach space, so each Hausdorff error is at most the distortion
$\delta$ of the correspondence. For $b\in B_r(Y)$ and
$\varepsilon>0$, choose $c\in C_Z$ with $\|b-c\|\le\delta+\varepsilon$.
Then
\[
 \dist(c,A_Z)\le \|c-b\|+\dist(b,A_Y)+d_{\hd}(A_Y,A_Z)
 \le r+2\delta+\varepsilon.
\]
The distance formula \eqref{eq:distance-neighborhood} applied to
$c\in C_Z$ gives
\[
 \dist(b,B_s(Z))
 \le \delta+\varepsilon+\max\{0,r-s+2\delta+\varepsilon\}.
\]
Let $\varepsilon\downarrow0$ and apply the same argument with the
roles reversed. This proves
\[
 d_{\hd}(B_r(Y),B_s(Z))\le3\delta+|r-s|.
\]
Also $|w_Y-w_Z|\le2\delta$ by the triangle inequality for Hausdorff
distance. Thus for $r=tw_Y$, $s=tw_Z$ the upper bound is at most
$5\delta$. Thus $d_{\gh}(G_t(Y),G_t(Z))\le5\delta$ because
$d_{\gh}\le d_{\hd}$ in this realization. Taking the infimum of
$\delta$ gives the factor $10$ in \eqref{eq:G-bounds}. The identity
correspondence on $Y$ with radii $tw_Y,sw_Y$ gives the time estimate.
\end{proof}

\subsection*{Step 3. Contract \(C_Y\) to \(B\)}

Define
\[
 K_t(Y)=
 \begin{cases}
  (1-t)C_Y,&0\le t<1,\\
  B,&t=1 .
 \end{cases}
\]
In the common space of \eqref{eq:width}, $V$ represents $K_1(Y)$.
We prove that
\begin{equation}\label{eq:K-bounds}
 d_{\gh}(K_t(Y),K_t(Z))\le2(1-t)d_{\gh}(Y,Z),\qquad
 d_{\gh}(K_t(Y),K_s(Y))\le D|t-s|.
\end{equation}
\begin{proof}
By \eqref{eq:sets}, a correspondence of distortion $\delta$ gives
$d_{\gh}(C_Y,C_Z)\le\delta$. Taking the infimum and rescaling by
$1-t$ gives the first bound for $t<1$. At $t=1$ both spaces equal $B$.

For the time estimate, first observe that $C_Y+V=C_Y$.
Indeed, for $c\in C_Y$ and $v\in V$, choose $c_m\in C_Y$ within
$D+1$ of $mv$. Then $(1-1/m)c+c_m/m\in C_Y$ and converges to
$c+v$. By construction $C_Y$ is closed in $E_Y$, and the affine
isometric image of the Banach space $E_Y$ is closed in the common
space by completeness. Hence the realized $C_Y$ is closed there.
This gives $C_Y+V\subset C_Y$. The opposite inclusion uses $0\in V$.

Fix $M>D$ and put $S=\{b\in C_Y:\|b\|\le M\}$. By \eqref{eq:width},
for every $c\in C_Y$ there is $v\in V$ with $\|c-v\|<M$.
Translation invariance gives $c-v\in S$. This proves
$C_Y\subset V+S$ and shows that $S$ is nonempty. Translation
invariance also gives $V+S\subset C_Y$, so $C_Y=V+S$.
For $t<1$ we have $(1-t)V=V$, and at $t=1$ we have
$V+(1-t)S=V$. Hence $V+(1-t)S$ represents $K_t(Y)$ for all $t$. Pairing points with the same $v\in V$ and
$b\in S$ gives
\[
 d_{\hd}\big(V+(1-t)S,V+(1-s)S\big)\le M|t-s|.
\]
Letting $M\downarrow D$ proves the time estimate.
\end{proof}

\paragraph{Joining the three steps.}
Since $\rho_B=d_B$ and $A_B=C_B$ is a linear copy of $B$, all three
maps fix $B$ in the cloud. The estimates in \eqref{eq:F-bounds}, \eqref{eq:G-bounds},
and \eqref{eq:K-bounds} show that each map sends spaces at zero distance
to spaces at zero distance. Taking $Z=B$ in the same estimates shows
that all intermediate spaces lie in $[B]$.

Set
\[
 \Phi(Y,t)=
 \begin{cases}
  F_{3t}(Y),&0\le t\le\tfrac13,\\
  G_{3t-1}(Y),&\tfrac13\le t\le\tfrac23,\\
  K_{3t-2}(Y),&\tfrac23\le t\le1 .
 \end{cases}
\]
The adjacent formulas agree at the joining times since
$F_1(Y)=G_0(Y)=A_Y$ and $G_1(Y)=K_0(Y)=C_Y$ in the cloud.

\paragraph{Joint continuity.}
The largest space Lipschitz bound is $10$. Before rescaling, the
three time bounds for fixed $Y$ are $\tfrac32D$, $w_Y\le2D$, and
$D$. Each interval is rescaled by a factor of $3$, so the time
bound on each third is $6D|t-s|$. If $s,t$ are in different thirds,
split the interval between them at the joining times and add the
bounds by the triangle inequality. Taking the infimum of $D$
over correspondences between $Y$ and $B$ gives a time bound of
$12d_{\gh}(Y,B)|t-s|$. Comparing first at the same time $t$ and then
at the same space $Y$ yields
\[
 d_{\gh}\bigl(\Phi(Z,t),\Phi(Y,s)\bigr)
 \le10d_{\gh}(Z,Y)+12d_{\gh}(Y,B)|t-s|.
\]
Thus $\Phi$ is jointly continuous, including at the joining
times and at the endpoints. The formulas at $t=0$ and $t=1$
give the required strong deformation retraction.

\paragraph{Example.}
Let $Y\subset B$ be a net, meaning that $Y$ is at finite Hausdorff
distance from $B$. Take the inclusion $i:Y\hookrightarrow B$ in the
definition of $\Omega_Y$. For $\omega\in\Omega_Y$,
\[
 k_y^o(\omega)=L_{i_*\omega}(o-y).
\]
Indeed, the normalized distance is $\Delta_{o-y}(z-o)$, and
the shift by $-o$ leaves its limit unchanged. Every
$\theta\in\Omega_B$ has the same functional as some $i_*\omega$.
To see this, choose $p:B\to Y$ with $\|p(z)-z\|$ uniformly bounded
and take $\omega=p_*\theta$. Admissibility and each limit
$L_\theta(v)$ are preserved by this bounded change. Lemma~\ref{lem:linear-profiles}
now identifies $A_Y$ isometrically with $\overline{Y-o}$ and
$C_Y$ with its closed convex hull in $B$. That hull contains $0$
and lies at finite Hausdorff distance from $B$. Applying the
translation argument of Step~3 inside $B$ shows that the hull is
$B$. In particular, for $Y=\Z\subset\R$, the middle step is
\[
 G_t(\Z)=\bigcup_{m\in\Z}[m-t/2,m+t/2]\qquad (0\le t\le1),
\]
which fills $\R$ at $t=1$.

\paragraph{Acknowledgements.}
The author is grateful to his advisor, A.~A.~Tuzhilin, and to A.~O.~Ivanov for their continued
interest in his work. He also thanks the participants of their seminar
\emph{Theory of Extremal Networks}.

OpenAI's ChatGPT was used in developing the proof and preparing the text.

\end{document}